\documentclass[12pt,leqno]{amsart}
\usepackage{times}
\usepackage{amsmath,amssymb,amscd,amsthm}
\usepackage{latexsym}
\usepackage{epsfig}
\usepackage{esint}

\newtheorem{theorem}{Theorem}
\newtheorem{lemma}{Lemma}
\newtheorem{proposition}{Proposition}

\newtheorem{corollary}{Corollary}

\newtheorem{claim}{Claim}

\newcommand{\f}[2]{\frac{#1}{#2}}

\newcommand{\tw}{\tilde{w}}
\newcommand{\tN}{\tilde{{\mathcal N}}}

\newcommand{\al}{\alpha}
\newcommand{\be}{\beta}
\newcommand{\ga}{\gamma}

\newcommand{\de}{\delta}

\newcommand{\ve}{\varepsilon}

\newcommand{\La}{\Lambda}
\newcommand{\si}{\sigma}

\newcommand{\vp}{\varphi}

\newcommand{\rone}{\mathbf R^1}

\newcommand{\cn}{{\mathcal N}}
\newcommand{\ct}{{\mathbf T}}

\newcommand{\cz}{\mathcal Z}

\newcommand{\cc}{\mathcal C}

\newcommand{\cl}{\mathcal L}
\newcommand{\cp}{\mathcal P}

\newcommand{\p}{\partial}

\newcommand{\beq}{\begin{equation}}
\newcommand{\eeq}{\end{equation}}
\newcommand{\beqna}{\begin{eqnarray*}}
\newcommand{\eeqna}{\end{eqnarray*}}
\newcommand{\beqn}{\begin{equation*}}
\newcommand{\eeqn}{\end{equation*}}
\newcommand{\bp}{\begin{proof}}
\newcommand{\ep}{\end{proof}}
\newcommand{\bprop}{\begin{proposition}}
\newcommand{\eprop}{\end{proposition}}
\newcommand{\bt}{\begin{theorem}}
\newcommand{\et}{\end{theorem}}
\newcommand{\bex}{\begin{Example}}
\newcommand{\eex}{\end{Example}}
\newcommand{\bc}{\begin{corollary}}
\newcommand{\ec}{\end{corollary}}
\newcommand{\bcl}{\begin{claim}}
\newcommand{\ecl}{\end{claim}}
\newcommand{\bl}{\begin{lemma}}
\newcommand{\el}{\end{lemma}}

\begin{document}

\title
[Local well-posedness  for the periodic Boussinesq equation ]
{Improved local well-posedness   for the periodic ``good'' Boussinesq equation }

\author{Seungly Oh, Atanas Stefanov}

\address{Seungly Oh, 
405, Snow Hall, 1460 Jayhawk Blvd. , 
Department of Mathematics,
University of Kansas,
Lawrence, KS~66045, USA}
\address{Atanas Stefanov, 405, Snow Hall, 1460 Jayhawk Blvd. , 
Department of Mathematics,
University of Kansas,
Lawrence, KS~66045, USA}
\date{\today}

\thanks{Oh and Stefanov are partially supported by  NSF-DMS 0908802 }



\begin{abstract}
 We prove that the ``good'' Boussinesq model is locally well-posed in the space 
 $H^{-\alpha}\times H^{-\al-2}$, $\al<\f{3}{8}$. In the proof, we employ the method of normal forms, which allows us to explicitly extract the rougher part of the solution, while we show that the remainder is in the smoother space $C([0,T], H^\be(\ct)),\  \be<\min(1-3\al, \f{1}{2}-\al)$.  Thus, we establish a smoothing effect of order $\min(1-2\al, \f{1}{2})$ for the nonlinear evolution. This is new even in the previously considered cases $\al \in (0, \f{1}{4})$. 
\end{abstract}

\maketitle
\date{today}

\section{Introduction}
 We consider the Cauchy problem for the periodic ``good'' Boussinesq problem 
 \begin{equation}
\label{1}
\left|
\begin{array}{l}
u_{tt}+u_{xxxx}-u_{xx}+(u^p)_{xx}=0, \ \ (t,x)\in \rone_+\times  \ct\\
u(0,x)=u_0(x); u_t(0,x)=u_1(x)
\end{array}\right.
\end{equation}  
This is a model  that was derived by Boussinesq, \cite{Bous}, in the case $p=2$ and  belongs to a family of Boussinesq  models, which all have the same level of formal validity. We will consider mostly the original model (i.e. with  $p=2$),  but we state some previous results in this generality for completeness. 

It was observed that \eqref{1} exhibits some desirable features, like local  well-posedness in various function spaces. Let us take the opportunity to explain the known results. Most of these results concern the same equation on the real line. 
It seems that the earliest work on the subject goes back to Bona and Sachs, who  have considered \eqref{1} and showed well posedness in 
$H^{\f{5}{2}+}(\rone)\times  H^{\f{3}{2}+}(\rone)$, \cite{Bona}. Interestingly, global well-posedness for \eqref{1} does not hold\footnote{except for small data, see below}, even if one requires smooth initial data with compact support. In fact, there are  ``instability by blow-up'' results for such unstable traveling waves for this equation. 

 Tsutsumi and Mathashi, \cite{TM}, established local well-posedness of \eqref{1}  in $H^1(\rone)\times H^{-1}(\rone)$. Linares lowered these  smoothness requirement 
  to $L^2(\rone)\times H^{-2}(\rone)$, $1<p<5$. In the same paper, Linares has showed the global existence of small solutions.  Farah, \cite{Farah} has shown well-posedness in $H^s(\rone)\times \tilde{H}^{s-2}(\rone)$, when $s>-1/4$ and the space $\tilde{H}^\al$ is defined via  $\tilde{H}^\al=\{u: u_x \in H^{\al-1}(\rone)\}$. Farah has also established ill-posedness (in the sense of lack of continuous dependence on initial data)  for all $s<-2$.  Kishimoto and Tsugava, \cite{KT}  have further improved this result to $s>-1/2$, which seems to be the most general result currently available for this problem. 
  
  Regarding the case of periodic boundary conditions, Fang and Grillakis, \cite{FG} who have established local well-posedness in $H^s(\ct)\times H^{s-2}(\ct)$, $s>0$ (when 
  $1<p<3$ in \eqref{1}).  This result was later improved to $s>-1/4$ for the quadratic equation by Farah and Scialom, \cite{FS}, by utilizing the optimal quadratic estimates   (proved in the paper) in the  Schr\"odinger $X^{s,b}$ spaces. In addition,  he showed  that these estimates fail below $s<-1/4$. Thus, local well-posedness for \eqref{1} in 
  $H^{-1/4+}$ is the best possible result, {\it obtainable by this method}. 
  
  Next, we point out that the initial value problem for the Boussinesq problem \eqref{1} is very closely related to the corresponding problem for quadratic Schr\"odinger equation 
  \begin{equation}
  \label{2}
  iu_t+u_{xx}+F(u, \bar{u})=0,
  \end{equation}
where $F$ is a bilinear form, which contains expressions in the form $u^2, u\bar{u}, \bar{u}^2$. Recall that Kenig, Ponce and Vega, \cite{CPV2} have established the local well-posedness 
 in $H^{-1/4+}(\rone)$ for \eqref{2}, while later Kishimoto-Tsugava, \cite{KT} (see also \cite{Nishimoto1}, \cite{Nishimoto2}) have established the sharpness of this result on the line (when the nonlinearity is $u\bar{u}$). 
 
 Our main concern in this paper is to extend the results of Farah and Scialom, \cite{FS} to even rougher initial data, namely in the class $H^s(\ct)\times H^{s-2}(\ct)$, $s>-3/8$. As we have mentioned above, the method of Farah and Scialom is optimal as far as the estimates are concerned. Our approach is similar to our earlier paper, \cite{OS}, where we apply the method of normal forms. The idea is that the roughest part of the solution to the nonlinear equation is the free solution, while the rest is actually much smoother. This allows us to obtain a smoothing estimate for the solution, in the sense described in our main result, Theorem \ref{theo:1} below. 
 
 Before we state our results, we introduce our setup. The Fourier transform and its inverse  in the time variable are defined via 
 $$
 \hat{h}(\tau)=\f{1}{2\pi} \int_{\rone} h(t) e^{-i t \tau} dt,  \ \ h(t)= \int_{\rone} \hat{h}(\tau) e^{i t \tau} d\tau 
 $$
Next,  we identify the torus $\ct$ with the interval $[0,2\pi]$.   In particular, for a smooth  $2\pi$-periodic 
 function $f:[0,2\pi]\to \cc$,  define the Fourier coefficients via 
 $$
\hat{f}(n)= \f{1}{2\pi} \int_0^{2\pi} f(x)e^{-i n x} dx.  
 $$
 The Sobolev spaces are defined for all $s\in \rone$ by the norms 
 $$
 \|f\|_{H^s(\ct)}=\| <n>^s \hat{f}(n)\|_{l^2}. 
 $$
 The Schr\"odinger  $X^\ve_{s,b}, \ve=\pm 1$ spaces, which will be relevant for our considerations, are defined via the completion of all sequences of Schwartz functions\footnote{Here, we take only those functions $F=(F_n)$, so that there exists $N$, so that $F_n(t) \equiv 0$, for all $|n|>N$.} $F=\{F_n\}_{n\in\cz \setminus \{ 0\} }$, $F_n:\rone\to \cc$, in the norm 
 \begin{equation}
 \label{30}
 \|F\|_{X^\ve_{s,b}}=\left(\sum_{n\in \mathbb{Z}\setminus \{0\}} \int_{\rone} (1+|\tau-\ve n^2|)^{2b} <n>^{2s}
 |\widehat{F_n}(\tau)|^2 d\tau\right)^{1/2}. 
 \end{equation}
In our definition \eqref{30}, we have restricted to the space of functions with spatial-mean zero (i.e. $\int_{\mathbf{T}} f(t,x)\,dx=0$ for all $t\in \mathbb{R}$).  We will justify this reduction in Section~\ref{sec:1.1}.

Observe that we have the duality relation $(X^{+}_{s,b})^* = X^{-}_{-s,-b}$.  
 \begin{theorem}
 \label{theo:1}
 Let $0<\al<3/8$ and $p=2$.  The Cauchy problem \eqref{1} is locally well-posed in $H^{-\al}(\ct)\times H^{-\al-2}(\ct)$.  That is, given $f\in H^{-\al}$ and $g\in H^{-\al-2}$, there exists $T := T(\|f\|_{H^{-\al}} +  \|g\|_{-\al-2})>0$ and a unique solution $u \in C^0_t ([0,T]; H^{-\al})$ of \eqref{1}.  
 
 Furthermore, we have the following representation formula for the solution $u$ 
 \begin{equation}
 \label{300}
u(t)= \f{1}{2\pi} \int_0^{2\pi} u_0  + e^{-(A_0 t+A_1 \f{t^2}{2})\cp} [\cos (t\sqrt{\p_x^4 - \p_x^2}) u_0 + \f{\sin(t \sqrt{\p_x^4 - \p_x^2})}{\sqrt{\p_x^4 - \p_x^2}} u_1] +z,
\end{equation} 
where $A_0=\f{1}{2\pi}\int_0^{2\pi} u_0(x)dx , A_1=\f{1}{\pi}\int_0^{2\pi} u_1(x) dx$, $\cp=\p_{xx}(\p_{x}^4-\p^2_x)^{-1/2}$ and \\ 
 $z\in  C^0_t([0,T]; H_x^{\be}(\mathbb{T}))$ for any $\be: \be <\min (1-3\al, \f{1}{2}-\al)$. 
 
 In particular, for data $(u_0, u_1)$, so that $\int_0^{2\pi} u_0(x)dx=0=\int_0^{2\pi} u_1(x) dx$, we have 
 $$
 u-\left[\cos (t\sqrt{\p_x^4 - \p_x^2}) u_0 + \f{\sin(t \sqrt{\p_x^4 - \p_x^2})}{\sqrt{\p_x^4 - \p_x^2}} u_1\right] \in C^0_t([0,T]; H_x^{\be}(\mathbb{T})). 
 $$
 \end{theorem}
  
 \section{Some preliminaries}
 \label{sec:1}
We first make a reduction of the Cauchy problem \eqref{1} to reduce to the case of mean value zero solutions, since this will be important for our argument. 
\subsection{Some reductions of the problem}\label{sec:1.1}
Observe that if 
$$
u(t,x)=\sum_{n=-\infty}^\infty  \hat{u}(t,n)e^{i n x},
$$
and if we consider the evolution of the zero mode, $\hat{u}(t,0)$, we find easily that 
$$
\f{d^2 \hat{u}(t,0)}{d t^2}=0. 
$$
Equivalently, integrating the equation in $x$ yields $\int_0^{2\pi} u_{tt}(t,x) dx=0$, whence  
$$
 \int_0^{2\pi} u(t,x) dx=\int_0^{2\pi} u(0,x) dx+ t\int_0^{2\pi} u_t(0,x) dx
$$
Thus, setting $w:u(t,x)=\f{1}{2\pi} \int_0^{2\pi} u(t,x) dx+ v(t,x)$, so that 
$$
 u(t,x)= \f{1}{2\pi}\left(\int_0^{2\pi} u(0,x) dx+ t\int_0^{2\pi} u_t(0,x) dx\right)+v(t,x)
$$ 
we conclude that $\int_0^{2\pi} v(t,x)dx =0$. Denoting 
$$
A(t)=\f{1}{2\pi}\left(\int_0^{2\pi} u(0,x) dx+ t\int_0^{2\pi} u_t(0,x) dx\right), 
$$
we see that  \eqref{1} 
is equivalent to the nonlinear problem 
\begin{equation}
\label{5}
\left|
\begin{array}{l}
v_{tt}+v_{xxxx}-v_{xx}+(A(t)+v)^2_{xx}=0 \\
v(0,x)=u_0(x)-\f{1}{2\pi} \int_0^{2\pi} u_0(x)dx; 
v_t(0,x)=u_1(x)-\f{1}{2\pi} \int_0^{2\pi} u_1(x)dx,
\end{array}
\right.
\end{equation}
We would like to consider the problem with data 
in the Sobolev spaces $H^{-\al}$, but to make our notations simpler, we prefer to work in $L^2(\ct)$, so we transform the equation \eqref{5} in $L^2(\ct)$ context. Namely, we introduce  $w= \langle \nabla \rangle^{-\alpha} v$, that is 
$$
w(t,x)=\sum_{n\neq 0} \f{\hat{v}(t,n)}{\langle n \rangle^{\alpha}} e^{i n x}. 
$$
Note that by construction  $\int_0^{2\pi} w(t,x)dx=0$. We can rewrite now 
 \eqref{5} as follows 
\begin{equation}\label{bous1}
\left| \begin{array}{l} w_{tt} - w_{xx} + w_{xxxx} + 2A(t)   w_{xx}+ \langle \nabla \rangle^{-\alpha} \p_x^2( \langle \nabla \rangle^{\alpha} w)^2 = 0, \qquad x\in \mathbf{T},\, t>0\\
w(0,x) = f(x) \in L^2(\mathbf{T});  \quad w_t (0,x) =  g(x) \in H^{-2}(\mathbf{T})\end{array} \right.
\end{equation}
where  
$$
f= \sum_{n\neq 0} \f{\widehat{u_0}(n)}{\langle n \rangle^{\alpha}} e^{i \pi n x}, 
g= \sum_{n\neq 0} \f{\widehat{u_1}(n)}{\langle n \rangle^{\alpha}} e^{i \pi n x}.
$$
 Note that $\int_\ct f(x)dx=0, \int_\ct g(x) dx=0$. 
 
 Set $L:= \sqrt{\p_x^4 - \p_x^2}$. Note that $\widehat{L h}(k)=|k|\sqrt{1+k^2} \hat{h}(k)$. Furthermore, in the space of functions with mean value zero, $L$ is invertible, with inverse given by 
 $$
L^{-1} h (x)=\sum_{k\neq 0} \f{1}{|k|\sqrt{1+k^2}} \hat{h}(k) e^{i k x}.
 $$
By the Duhamel's principle, \eqref{bous1} is equivalent to 
\begin{align}\label{duh}
w(t,x) &= \cos (t L) f(x) +  \sin (t L )[L^{-1} g]\\ \notag
& + \int_0^t \sin ((t-s) L) L^{-1}[2 A(s)w_{xx}+ \langle \nabla \rangle^{-\alpha} 
\p_x^2( \langle \nabla \rangle^{\alpha} w(s,\cdot))^2\, ]ds.
\end{align}  
Using Euler's formula, we can  write $w = w^+ + w^-$,  where
\begin{align*}
w^+ (t,x) &= \frac{e^{itL} f}{2} + \frac{e^{itL} L^{-1} g}{2i} + \frac{1}{2i} \int_0^t e^{i(t-s)L}[F(w^+) + \mathcal{N}(w^+ + w^-, w^+ + w^-)]\, ds\\
w^- (t,x) &= \frac{e^{-itL} f}{2} - \frac{e^{-itL} L^{-1} g}{2i}  - \frac{1}{2i} \int_0^t e^{-i(t-s)L}[F(w^-)+\mathcal{N} (w^+ + w^-, w^+ + w^-)]\, ds
\end{align*}
where $F(w)=2 A(s) L^{-1}\p_{xx} w$ and $ \mathcal{N}(u,v) := L^{-1} \langle \nabla \rangle^{-\alpha} \p_x^2 ( \langle \nabla \rangle^{\alpha} u \langle \nabla \rangle^{\alpha} v)$. Thus, we have replaced the single wave equation for $w$ into a system of equations, involving $w^+, w^-$. Namely, denoting $\mathcal{L}(f,g):= \frac{1}{2}e^{itL} f + \frac{1}{2i} e^{itL} L^{-1} g$ (or $\mathcal{L}$ for short), we have 
\begin{equation}
\label{100}
\left\{
\begin{array}{ll}
(\p_t-i L)w^+   &= F(w^+)+  \mathcal{N}(w^+ + w^-, w^+ + w^-), \\
(\p_t+i L)w^-  &= F(w^-)+ \mathcal{N} (w^+ + w^-, w^+ + w^-) ,\\
w^+(0,x) &=\mathcal{L}(0)=\f{1}{2} f+ \f{1}{2i} L^{-1} g \in L^2\\
 w^-(0,x) &=\bar{\cl}(0)=\f{1}{2} \bar{f}- \f{1}{2i} L^{-1} \bar{g} \in L^2
\end{array}
\right.
\end{equation}
The term $F(w^\pm)$ creates certain complications, mostly of technical nature, which we now address. Write 
$$
F(w)(s)=2A(s)L^{-1}\p_{xx}=(A_0+s A_1)\cp w,
$$
where $A_0=\f{1}{\pi}\int_0^{2\pi} u_0(x) dx, A_1=\f{1}{\pi}\int_0^{2\pi} u_1(x) dx$ are scalars and $\cp:=L^{-1} \p_{xx}$ is an order zero differential operator, given by the symbol $-\f{|k|}{<k>}$ and hence bounded on all Sobolev spaces. We now resolve the inhomogeneous equation 
$
(\p_t-i L- F)w^+=G
$  (for any right hand side $G$) 
in the following way. Introduce 
$$
w^\pm(s)=e^{(A_0 s+A_1 \f{s^2}{2})\cp} \tw^\pm(s), 
$$
 where $e^{(A_0 s+A_1 \f{s^2}{2})\cp} $ is a bounded operator on any $L^2$ based Sobolev space, which can be represented for example by its power series. We have 
\begin{eqnarray*}
(\p_t-i L) w^+ &=& e^{(A_0 t+A_1 \f{t^2}{2})\cp} (\p_t-i L) \tw^+ + (A_0+t A_1)\cp e^{(A_0 t+A_1 \f{t^2}{2})\cp} \tw^+=\\
&=& e^{(A_0 t+A_1 \f{t^2}{2})\cp} (\p_t-i L) \tw^+ + F[w^+].
\end{eqnarray*}
Thus, 
$$
G=(\p_t-i L-F) w^+=e^{(A_0 t+A_1 \f{t^2}{2})\cp} (\p_t-i L) \tw^+,
$$
whence\footnote{Note that $e^{-(A_0 t+A_1 \f{t^2}{2})\cp}$ is the (bounded) inverse of $e^{(A_0 t+A_1 \f{t^2}{2})\cp}$}
$$
(\p_t-i L) \tw^+=e^{-(A_0 t+A_1 \f{t^2}{2})\cp} G. 
$$
Similar computations work for $w^-$. 
 Thus, we have reduced \eqref{100} to the following equation for $\tw^+$ 
 \begin{equation}
 \label{104}
 (\p_t-i L) \tw^+=e^{-(A_0 t+A_1 \f{t^2}{2})\cp} 
 \mathcal{N}(e^{(A_0 t+A_1 \f{t^2}{2})\cp}(\tw^+ + \tw^-), e^{(A_0 t+A_1 \f{t^2}{2})\cp}(\tw^+ + \tw^-)),
 \end{equation}
 and similar for $w^-$. Observe that $\tw^+(0)=\cl(0)$ and $\tw^-(0)=
\bar{\cl}(0)$. For convenience, introduce the notation 
\begin{equation}
\label{250}
\tilde{\mathcal{N}}(u,v):=e^{-(A_0 t+A_1 \f{t^2}{2})\cp} 
 \mathcal{N}(e^{(A_0 t+A_1 \f{t^2}{2})\cp}u, e^{(A_0 t+A_1 \f{t^2}{2})\cp}v),
\end{equation}
so that our main governing equation \eqref{104}, now takes the form 
$$
(\p_t-i L) \tw^+=
\tilde{\mathcal{N}}(\tw^+ + \tw^-,\tw^+ + \tw^- )
$$

We  note that the operators $e^{\pm(A_0 t+A_1 \f{t^2}{2})\cp}$ are mostly harmless, in the sense that they are bounded on all function spaces considered in the paper. At first  reading, the reader may as well assume that $A_0=A_1=0$ (which corresponds to the important case of mean value zero data) to avoid the cumbersome technical complications. 

\subsection{Construction of the normal forms: the case with mean value zero}
\label{2.2}
We start with the case $A_0=A_1=0$ in order to simplify matters. In the next section, we indicate how to handle the general case. 

  Clearly, we have  
\begin{equation}
\label{10}
\|\mathcal{L}(f,g)\|_{L^2(\mathbf{T})}\leq  \frac{1}{2}\|f\|_{L^2} + \frac{1}{2} \| |\nabla|^{-1} \langle \nabla\rangle^{-1} g\|_{L^2} \sim \|f\|_{L^2} + \|g\|_{H^{-2}}. 
\end{equation} 
We introduce further variables $z^\pm$, so that 
 $w^+=\cl +z^+, w^-=\bar{\cl}+z^-$.  This yields a new set of two equations for the unknowns $z^\pm$. Furthermore, the nonlinearities 
  take one of the following forms:
$$
\mathcal{N}(\mathcal{L}, \mathcal{L}), \quad
\mathcal{N}(\mathcal{L}, \overline{\mathcal{L}}),\quad 
\mathcal{N}(\overline{\mathcal{L}}, \overline{\mathcal{L}}),\quad
\mathcal{N}(\mathcal{L}, z^{\pm}), \quad
\mathcal{N}(\overline{\mathcal{L}}, z^{\pm}),\quad
\mathcal{N}(z^{\pm}, z^{\pm}).
$$
  We   construct an explicit solution, in the form of a bilinear pseudo-differential operator (i.e. a ``normal form''), which will take care of    the first three non-linearities, that is those in the form $\mathcal{N}(\mathcal{L}, \mathcal{L}), \ 
\mathcal{N}(\mathcal{L}, \overline{\mathcal{L}}),\ 
\mathcal{N}(\overline{\mathcal{L}}, \overline{\mathcal{L}})$. That is, we are looking to solve for $\ve=\pm 1$, 
 \begin{equation}
 \label{eqh}
  (\p_t - i \ve \ L) h^\ve = \frac{1}{2i}\left[\mathcal{N}(\mathcal{L}, \mathcal{L}) +
2\mathcal{N}(\mathcal{L}, \overline{\mathcal{L}}) +
\mathcal{N}(\overline{\mathcal{L}},\overline{\mathcal{L}})\right].
\end{equation}
 In order to prepare us for our choice of $h^\ve$, we need to display some algebraic relations for the symbols. More precisely, for $\ve, \ve_1, \ve_2 \in \{ -1, 1\}$, we have 
\begin{align*}
(\tau + \omega) - \ve \sqrt{(\xi+ \eta)^4 + (\xi+ \eta)^2} &= (\tau - \ve_1 \sqrt{\xi^4 + \xi^2} ) + (\omega - \ve_2\sqrt{\eta^4 + \eta^2})\\
	& + \ve_1|\xi| \langle \xi \rangle + \ve_2|\eta|\langle \eta\rangle - \ve |\xi + \eta| \langle \xi + \eta \rangle.
\end{align*}
which implies that for every bilinear pseudo-differential operator $\La_\si$ with symbol $\si(\xi, \eta)$, that is 
$\La_\si(u,v)=\sum_{\xi, \eta \in \cz}\si(\xi, \eta) \hat{u}(\xi) \hat{v}(\eta)e^{i(\xi+\eta)x}$, we have 
\begin{eqnarray*}
& & (\p_t-i L)\La_\si(u,v)=-i(\La_\si((\p_t-i L)u, v)+\La_\si(u,(\p_t-i L) v)+ \La_\mu(u,v)),\\
& & \mu(\xi, \eta)=\si(\xi, \eta)(\ve_1|\xi| \langle \xi \rangle + \ve_2|\eta|\langle \eta\rangle - \ve |\xi + \eta| \langle \xi + \eta \rangle). 
\end{eqnarray*}
In particular, if $u,v$ are free solutions, i.e. $(\p_t-i L)u=(\p_t-i L)v=0$, we get 
$$
(\p_t-i L)\La_\si(u,v)=-i \La_\mu(u,v).
$$
 Thus, we define a bilinear pseudo-differential operator $T$ by the formula 
\begin{equation}\label{normal}
T^{\ve; \ve_1, \ve_2}(u,v)(x) := -\frac{1}{2} \sum_{\xi\eta (\xi+ \eta)\neq 0} 
\frac{|\xi+ \eta| \langle \xi\rangle^{\alpha} \langle \eta \rangle^{\alpha} \widehat{u}(\xi) \widehat{v}(\eta) \, e^{i(\xi+ \eta)x}}{\langle \xi+\eta\rangle^{1+\alpha} [\ve_1 |\xi| \langle \xi \rangle + \ve_2 |\eta|\langle \eta\rangle - \ve|\xi + \eta| \langle \xi + \eta \rangle]}
\end{equation}
we get that 
\begin{eqnarray*}
(\p_t-i \ve L) T^{\ve;+,+}(\cl, \cl)=\f{1}{2 i} \cn(\cl, \cl), \\
(\p_t-i \ve L) T^{\ve;+,-}(\cl, \bar{\cl})=\f{1}{2 i} \cn(\cl, \bar{\cl})\\
(\p_t-i \ve L) T^{\ve;-,-}(\bar{\cl}, \bar{\cl})=\f{1}{2 i} \cn(\bar{\cl}, \bar{\cl}), 
\end{eqnarray*}
which allows us to get a solution of \eqref{eqh} in the form 
 \begin{equation}
 \label{15}
 h^\ve = T^{\ve; +,+}(\mathcal{L}, \mathcal{L}) +
2T^{\ve;+,-}(\mathcal{L}, \overline{\mathcal{L}}) +
T^{\ve;-,-}(\overline{\mathcal{L}},\overline{\mathcal{L}}).
 \end{equation}
 We perform another change of variables,    
 $\Psi^\pm: z^\pm=h^\pm +\Psi^\pm$, so that 
 \begin{equation}
 \label{20}
\left| \begin{array}{l} (\p_t - i L)\Psi^+ =  \mathcal{N}(\mathcal{L}+\overline{\mathcal{L}}, h^\pm + \Psi^{\pm})  + \mathcal{N} (h^\pm +\Psi^{\pm}, h^\pm +\Psi^{\pm})\\
	\Psi^+(0,x)=-[T^{+,+}(\mathcal{L}, \mathcal{L}) +
2T^{+,-}(\mathcal{L}, \overline{\mathcal{L}}) +
T^{-,-}(\overline{\mathcal{L}},\overline{\mathcal{L}})]|_{t=0} , 
	\end{array} \right.
\end{equation}
similar formula holds for  $\Psi^-$. In fact, from now on, we will set $\ve=+1$, since the case $\ve=-1$ can always be reduced to the case $\ve=+1$. Thus, we drop $\ve$ from our notations, for example $T^{\ve_1, \ve_2}$ is used to denote $T^{+1;\ve_1, \ve_2}$ etc. 

With that, we have largely prepared the nonlinear problem to its current form \eqref{20}. Note that by our construction, $\Psi^\pm$ is  a mean value zero function. 
\subsection{Construction of the normal forms: the general case}
\label{2.3}
In the general case, and  having in mind the particular form of the right-hand side of \eqref{104}, we set $\tw^+=\cl+z^+, \tw^-=\bar{\cl}+z^-$. Note $z^\pm(0)=0$. Similar to \eqref{15}, set 
\begin{eqnarray*}
h^\ve &=&  e^{-(A_0 t+A_1 \f{t^2}{2})\cp} T^{\ve; +,+}(e^{(A_0 t+A_1 \f{t^2}{2})\cp} \mathcal{L}, e^{(A_0 t+A_1 \f{t^2}{2})\cp}\mathcal{L}) +\\
&+& 2 e^{-(A_0 t+A_1 \f{t^2}{2})\cp} T^{\ve;+,-}(e^{(A_0 t+A_1 \f{t^2}{2})\cp} \mathcal{L}, e^{(A_0 t+A_1 \f{t^2}{2})\cp}  \overline{\mathcal{L}}) +\\
&+& e^{-(A_0 t+A_1 \f{t^2}{2})\cp} T^{\ve;-,-}(e^{(A_0 t+A_1 \f{t^2}{2})\cp} \overline{\mathcal{L}},e^{(A_0 t+A_1 \f{t^2}{2})\cp}  \overline{\mathcal{L}}).
\end{eqnarray*}
With this assignment for $h^\ve$, we will certainly not 
get the nice exact identity \eqref{eqh}. However, we get something similar (up to an error term), which is good enough for our purposes. Namely, 
\begin{eqnarray*}
(\p_t - i \ve \ L) h^\ve &=&  \tilde{\mathcal N}(\cl, \cl)+  
2 \tilde{\mathcal N}(\cl, \bar{\cl})+\tilde{\mathcal N}(\bar{\cl}, \bar{\cl})
 +Err,
\end{eqnarray*}
where the error term contains all the terms obtained when the time derivative hits the terms 
$e^{\pm(A_0 t+A_1 \f{t^2}{2})\cp}$ in the formula for $h^\ve$. Thus, a typical error term will be 
\begin{equation}
\label{108}
Err\sim e^{-(A_0 t+A_1 \f{t^2}{2})\cp}(-A_0-A_1 t) \cp [T^{\ve; +,+}(e^{(A_0 t+A_1 \f{t^2}{2})\cp} \mathcal{L}, e^{(A_0 t+A_1 \f{t^2}{2})\cp}\mathcal{L})]. 
\end{equation}
Similar to Section \ref{2.2} above, introduce the new variables $\Psi^\pm$, so that 
$z^\pm=h^\pm+\Psi^\pm$. That is, $\tw^+=\cl+h^+ + \Psi^+$, $\tw^-=\bar{\cl}+
h^- + \Psi^-$. We obtain the following equations for $\Psi^\pm$ (note the similarity to \eqref{20}) 
\begin{equation}
\label{220}
\left| \begin{array}{l} (\p_t - i L)\Psi^+ =  \mathcal{\tilde{N}}(\mathcal{L}+\overline{\mathcal{L}}, h^\pm + \Psi^{\pm})  + \mathcal{\tilde{N}} (h^\pm +\Psi^{\pm}, h^\pm +\Psi^{\pm})-Err\\
	\Psi^+(0,x)=-[T^{+,+}(\mathcal{L}, \mathcal{L}) +
2T^{+,-}(\mathcal{L}, \overline{\mathcal{L}}) +
T^{-,-}(\overline{\mathcal{L}},\overline{\mathcal{L}})]|_{t=0}, 
	\end{array} \right.
\end{equation}
Note that for the initial data, that is at $t=0$, 
$$
e^{-(A_0 t+A_1 \f{t^2}{2})\cp} T^{+,+}(e^{(A_0 t+A_1 \f{t^2}{2})\cp} \mathcal{L}, e^{(A_0 t+A_1 \f{t^2}{2})\cp}\mathcal{L})|_{t=0}=T^{+,+}(\mathcal{L}, \mathcal{L}) |_{t=0}. 
$$
etc. whence we get the same initial conditions in \eqref{220} and \eqref{20}. Thus, our equation \eqref{220} will be the main object of interest for the remainder of the paper. 

\subsection{$X^{s,b}$ estimates and embeddings} 
 We now need to state the relevant {\it a priori} estimates for the linear problem
\begin{lemma}
\label{le:5}
Let $m$ solve the linear inhomogeneous problem 
$$
(\p_t - i \ve L)m=F, m(0)=m_0.
$$
Then, for all $T>0$, $s\in\rone$ and $b>1/2$, we have for all cut-off 
functions $\eta\in C^\infty_0$
\begin{equation}
\label{25}
\|\eta(t)m\|_{X^\ve_{s,b}}\leq C_\eta(\|m_0\|_{H^s}+\|F\|_{X^\ve_{s,b-1}}). 
\end{equation}
\end{lemma}
\begin{proof}
The proof is essentially contained in Proposition 3.12 in Tao's book, \cite{Tao2}. More precisely, Proposition 3.12 in \cite{Tao2} establishes estimates like \eqref{25} for arbitrary dispersion relations.  As a result of it, we have  
\begin{equation}
\label{40}
\|\eta(t)m\|_{Y^\ve_{s,b}}\leq C_{\vp} (\|m_0\|_{H^s}+\|F\|_{Y^\ve_{s,b-1}}),
\end{equation}
where  
$$
\|F\|_{Y^\ve_{s,b}}=\left(\int_{\rone} \sum_{n\in \mathbb{Z}\setminus \{0\}} (1+|\tau-\ve |n|\langle n\rangle|)^{2b} \langle n\rangle^{2s}
 |\widehat{F_n}(\tau)|^2 d\tau\right)^{1/2}. 
 $$
 The difference between \eqref{40} and the estimate \eqref{25} is that we insist on using the standard Schr\"odinger $X_{s,b}$ spaces, instead of the less standard $Y_{s,b}$ spaces. But in fact, the two spaces are equivalent. That is, we claim that the symbols are equivalent in the following sense.  More precisely, since  $0<|n|\langle n\rangle -n^2<1$, we have that the two norms $\|\cdot\|_{Y_{s,b}^\ve}$ and $\|\cdot\|_{X_{s,b}^\ve}$ are equivalent (for all values of the parameters $\ve, s,b$) and hence \eqref{40} is equivalent to \eqref{25}, and hence \eqref{25} is established.   
\end{proof}
Next, there is the following important  embedding result, due to Bourgain, \cite{Bour}. 
\begin{lemma}
The following embeddings hold:  $X^\pm_{0,\frac{3}{8}} \subset L^4_{t,x}$ and   $X^\pm_{0+, \frac{1}{2}+} \subset L^6_{t,x}$.  
\end{lemma}
The stability of the $X^\ve_{s,b}$ norms with respect to products with smooth functions is the following standard 
\begin{lemma} 
\label{le:11}
For a cut-off functions $\eta\in C^\infty_0$, there is $C=C_\eta$, so that 
$$
\|\eta(t) m\|_{X^\ve_{s,b}}\leq C \|m\|_{X^\ve_{s,b}}.
$$
\end{lemma}
Lemma \ref{le:11} appears as  Lemma 2.11 in \cite{Tao2}. From the proof of Lemma \ref{le:11}, it can be inferred that for $b\in (1/2,1)$, one can select 
$C_\eta= C (\|\eta\|_{L^1(\rone)}+\|\eta''\|_{L^1(\rone)})$ for some absolute constant $C$. 

As a  corollary, we derive the following estimate, which will be useful for us in the sequel 
\begin{equation}
\label{200}
\|\eta(t) e^{(A t+B t^2)\cp} m\|_{X^\ve_{s,b}}\leq C_{\eta, A, B} \|m\|_{X^\ve_{s,b}}.
\end{equation}
For the proof of \eqref{200}, take more generally a $C^2$ function $g(t)$ instead of $A t+B t^2$. One may expand the operator $e^{g(t)\cp}$ in power series 
$$
e^{g(t)\cp}=\sum_{k=0}^\infty \f{g(t)^k \cp^k}{k!}.
$$
Thus, given that $\|\cp\|\leq 1$, it is enough to show that  
$\|\eta(t)g(t)^k m\|_{X^\ve_{s,b}}\leq C_{k} \|m\|_{X^\ve_{s,b}}$, so that 
$\sum_k \f{C_k}{k!}<\infty$. By the remark above, one could take 
$$
C_k=C(\|\eta(t)g(t)^k\|_{L^1(\rone)}\|+\|(\eta(t)g(t)^k)''\|_{L^1(\rone)}
\leq C  k^2(1+\|g\|_{C^2(-M,M)})^k
$$
where $supp\eta\subset (-M,M)$. Since $\sum^\infty_{k=1} \f{k^2(1+\|g\|_{C^2(-M,M)})^k}{k!}<\infty$, \eqref{200} is established.

 \section{Proof of Theorem \ref{theo:1}} 
After the preparatory Section \ref{sec:1}, we are ready to take on the proof of  Theorem \ref{theo:1}. Let us recapitulate what we have done so far. 
First, we have represented the original problem in the form of \eqref{bous1}, which concern  mean value zero $L^2$ solutions, that is we need to show well-posedness for $L^2\times H^{-2}$ data for the problem \eqref{bous1}.  Next, instead of considering the second order in time equation, we have reduced to the first order in time system of equations for $w^\pm$, \eqref{100}. By an additional change of variables, 
this was replaced by the system \eqref{104} for the slightly modified 
$\tw^\pm$. Next, we have constructed in Section \ref{2.3} explicitly a solution $h^\pm$ 
to the linear inhomogeneous system with right hand sides involving the free solutions. That is, 
$$
\tw^+=\cl+z^+=\cl+h^+ + \Psi^+; \ \ \tw^-=\bar{\cl}+z^-=\bar{\cl}+h^- + \Psi^-.
$$
In terms of $w^\pm$
\begin{equation}
\label{305}
w^+=e^{-(A_0 t+A_1 \f{t^2}{2})\cp}[\cl+h^+ + \Psi^+]; \ \ 
w^-=e^{-(A_0 t+A_1 \f{t^2}{2})\cp}[\bar{\cl}+h^- + \Psi^-].
\end{equation}
Given that, as we pointed out earlier, the operators $e^{(A_0 t+A_1 \f{t^2}{2})\cp}$ are harmless (i.e. they preserve the relevant function spaces) and the explicit structure of $\cl, h^\pm$,  it now remains to resolve the nonlinear equation for $\Psi^\pm$,  \eqref{220}. We will do that, as we have indicated earlier, in the spaces $X^\pm_{\ga, \f{1}{2}+}$, where $\ga<\min (\f{1}{2}, 1-2\al)$. 

Our next lemma shows that the initial data $\Psi^+(0,x)$  is $H^1$ smooth. 
\begin{lemma}
For $0<\alpha <1/2$ and $\ve_1, \ve_2\in \{+1, -1\}$, we have $T^{\ve_1, \ve_2}: L^2 \times L^2 \rightarrow H^{1}$
\end{lemma}

\begin{proof}
We define the symbols $\sigma^{\ve_1, \ve_2}$ based on the expression \eqref{normal} so that  
\begin{align*}
T^{\ve_1, \ve_2} (u,v)(x) &= \sum_{\xi, \eta \in \mathbb{Z}} \sigma^{\ve_1,\ve_2} (\xi, \eta) \widehat{u}(\xi) \widehat{v}(\eta) e^{i(\xi+\eta)x}\\
 &= \sum_{\xi \in \mathbb{Z}} \left[\sum_{\eta\in \mathbb{Z}} \sigma^{\ve_1, \ve_2} (\xi - \eta, \eta) \widehat{u}(\xi-\eta) \widehat{v}(\eta)\right] e^{i\xi x}.
\end{align*}
  
Note from the sum in \eqref{normal} that $\sigma^{\ve_1, \ve_2}\equiv 0$ if $\xi \eta (\xi+ \eta) = 0$.  Otherwise, we have
\begin{align*}
\sigma^{-, -}( \xi, \eta) &\sim \frac{ \langle \xi\rangle^{\alpha}\langle \eta \rangle^{\alpha} }{\langle \xi+\eta \rangle^{\alpha}\max (\xi^2, \eta^2)};\\
\sigma^{+, +}( \xi, \eta) &\sim \frac{ 1 }{\langle \xi+\eta \rangle^{\alpha}\langle \xi\rangle^{1-\alpha}\langle \eta \rangle^{1-\alpha}};\\
\sigma^{+, -}( \xi, \eta) &\sim \frac{  \langle \xi\rangle^{\alpha} }{\langle \xi+\eta \rangle^{\alpha+1}\langle \eta \rangle^{1-\alpha}}.
\end{align*}

The following estimates are based on the size of symbols $\sigma^{\pm, \pm}$.  This is justified by taking absolute values on the Fourier side.

Let $u,v \in L^2 (\mathbf{T})$.  Then
\begin{align*}
\|T^{+,+}(u,v)\|_{H^{1}} &\sim \|\sum_{\eta\in \mathbf{Z}}  \frac{\langle \xi\rangle^{1-\alpha} } {\langle \xi -\eta \rangle^{1-\alpha}\langle \eta \rangle^{1-\alpha}}\widehat{u}(\xi- \eta) \widehat{v}(\eta) \|_{L^2_{\xi} (\mathbf{Z})}\\
&\lesssim \|\sum_{|\eta|\ll |\xi|}  \frac{\widehat{u}(\xi- \eta) \widehat{v}(\eta)} { \langle \eta \rangle^{1-\alpha}} \|_{L^2_{\xi}} + \| \sum_{|\eta | \gtrsim |\xi|}\frac{ \widehat{u}(\xi- \eta) \widehat{v}(\eta) } {\langle \xi-\eta \rangle^{1-\alpha} }   \|_{L^2_{\xi}}\\ 
&\lesssim \|\widehat{u}\|_{L^2_{\xi}}\sum_{\eta\in \mathbf{Z}}  \frac{|\widehat{v}|(\eta)} {\langle \eta \rangle^{1-\alpha}}   +  \|\frac{ \widehat{u}(\cdot)}{\langle \cdot\rangle^{1-\alpha}} \|_{L^1_{\xi}} \|\widehat{v}\|_{L^2_{\xi}}\\
&\lesssim \|u\|_{L^2(\mathbf{T})} \|v\|_{L^2(\mathbf{T})}.
\end{align*}

\begin{align*} 
\|T^{+,-}(u,v)\|_{H^{1}} &\sim \|\sum_{\eta\in \mathbf{Z}}  \frac{ \langle \xi -\eta \rangle^{\alpha} } {\langle \xi\rangle^{\alpha} \langle \eta \rangle^{1-\alpha}}\widehat{u}(\xi- \eta) \widehat{v}(\eta) \|_{L^2_{\xi} (\mathbf{Z})}\\
&\lesssim \|\sum_{|\eta|\ll |\xi|}  \frac{1} { \langle \eta \rangle^{1-\alpha}}\widehat{u}(\xi- \eta) \widehat{v}(\eta) \|_{L^2_{\xi}} + \| \frac{ 1 } {\langle \xi\rangle^{\alpha}}\sum_{|\eta | \gtrsim |\xi|}  \frac{\widehat{u}(\xi- \eta) \widehat{v}(\eta)}{ \langle \xi-\eta\rangle^{1-2\alpha}} \|_{L^2_{\xi}}\\ 
&\lesssim \|\widehat{u}\|_{L^2_{\xi}}\sum_{\eta\in \mathbf{Z}}  \frac{|\widehat{v}|(\eta)} {\langle \eta \rangle^{1-\alpha}}   +  
 \| \langle \nabla \rangle^{-\alpha}[v\cdot \langle \nabla \rangle^{2\alpha-1}u] \|_{L^2_{x}(\mathbf{T})}
 \\ 
&\lesssim \|u\|_{L^2(\mathbf{T})} \|v\|_{L^2(\mathbf{T})}
\end{align*}
where we have used Sobolev embedding and H\"older's inequality to obtain
\begin{align*}
 \| \langle \nabla \rangle^{-\alpha}[ v\cdot \langle \nabla \rangle^{2\alpha-1}u ]\|_{L^2_{x}(\mathbf{T})} &\lesssim \| v\cdot \langle \nabla \rangle^{2\alpha-1}u\|_{L^{\frac{2}{2\alpha+1}}_x(\mathbf{T})}\\
 	&\lesssim \|v\|_{L^2_x(\mathbf{T})} \|\langle \nabla \rangle^{2\alpha-1}u\|_{L^{\frac{1}{\alpha}}_x(\mathbf{T})}\\
 	&\lesssim \|u\|_{L^2_x(\mathbf{T})} \|v\|_{L^2_x(\mathbf{T})}.
 \end{align*}
The estimate for $T^{-,-}$ follows from the fact $\sigma^{-,-}\leq \sigma^{+,+}$ and we are done. 
\end{proof}
 
 \subsection{Reducing the proof to bilinear and trilinear estimates}

Assume for a moment that for some $\gamma>0$, $\Psi^+ \in X^+_{\gamma, 1/2+}$.  Then in the equation~\eqref{20} for $\Psi^+$, the right-hand side consists of   nonlinearities in the form $\mathcal{\tilde{N}}(u,v)$ where 
$$
(u,v) \in [X^{\pm}_{\gamma,\frac{1}{2}+} \times X^{\pm}_{0,\frac{1}{2}+}] \cup \left[L^{\infty}_t H^1_x \times L^{\infty}_t H^1_x \right] \cup [ L^{\infty}_t H^1_x \times X^{\pm}_{0,\frac{1}{2}+}].
$$
Therefore, in order to prove the theorem (as a result of a contraction argument in 
$X^+_{\ga,\f{1}{2}+}$), we need to control the nonlinear terms in appropriate norms. More precisely, we shall need following estimates for $\ve_1, \ve_2 \in \{ -1, 1\}$ in order to proceed with the standard contraction argument:
\begin{eqnarray}
\label{105}
\|\mathcal{\tN} (u,v)\|_{X^{+}_{\gamma, -\frac{1}{2}+}} &\lesssim \|u\|_{X^{\ve_1}_{\gamma, \frac{1}{2}+}} \|v\|_{X^{\ve_2}_{0,\frac{1}{2}+}}\\
\label{110}
\|\mathcal{\tN} (u,v)\|_{X^{+}_{\gamma, -\frac{1}{2}+}} &\lesssim \|u\|_{L^{\infty}_t H^1_x} \|v\|_{L^{\infty}_t H^1_x}.
\end{eqnarray}
In addition, we would have liked to have 
\begin{equation}
\label{115}
\|\mathcal{\tN} (u,v)\|_{X^{+}_{\gamma, -\frac{1}{2}+}} \lesssim \|u\|_{L^{\infty}_t H^1_x} \|v\|_{X^{\ve_1}_{0,\frac{1}{2}+}}
\end{equation}
but {\it this  estimate  turns out to be false}. On the other hand, the entry $u$ is not just an arbitrary $L^\infty_t H^1_x$ function, but rather a bilinear expression in the form $T^{\ve_1, \ve_2}(e^{\pm it L} f, e^{\pm it L} g)$. Due to this fact, we replace \eqref{115} with a {\it tri-linear} estimate, see Lemma \ref{le:30} below. 

We also make the observation that in what follows, we can replace $\tN$ by $\cn$. Indeed, referring to \eqref{250} and taking into account that $e^{(A_0 t+A_1 \f{t^2}{2})\cp}$ preserves $X^\pm_{s,b}$,  we have 
$$
\|\eta(t) \tN(u,v)\|_{X^\pm_{s,b}}\leq C_\eta \|\mathcal{N}(e^{(A_0 t+A_1 \f{t^2}{2})\cp}u, e^{(A_0 t+A_1 \f{t^2}{2})\cp}v)\|_{X^\pm_{s,b}}.
$$
Note that for $\tilde{u}=e^{(A_0 t+A_1 \f{t^2}{2})\cp}u$, 
we have from \eqref{200} that  $\|\tilde{u}\|_X\leq \|u\|_X$ for all function spaces that appear in \eqref{105} and \eqref{110} and hence, it suffices to establish \eqref{105} and \eqref{110}  with $\tN$ replaced by ${\mathcal N}$. 

We state the following results, which will be our main technical tools in order to finish the proof of Theorem \ref{theo:1}. In them, we assume $\gamma \geq 0$.

Our next lemma is a proof of \eqref{105}. 
\begin{lemma}
\label{le:20}
For $u,v$ smooth and $0\leq\alpha<1/2$, let $\ga$ be such that $2\alpha -1/2 < \gamma< 1/2$. Then 
$$
\|\mathcal{N} (u,v)\|_{X^{+}_{\gamma, -\frac{1}{2}+}} \lesssim \|u\|_{X^{\ve_1}_{\gamma, \frac{1}{2}+}} \|v\|_{X^{\ve_2}_{0,\frac{1}{2}+}}.
$$
\end{lemma}
 The next lemma concerns \eqref{110}. More precisely, we have 
 \begin{lemma}
 \label{le:25}
For $u,v$ smooth and $0\leq\alpha<1/2$, let $\gamma: \,\gamma<1/2$. 
$$
\|\mathcal{N} (u,v)\|_{X^{+}_{\gamma, -\frac{1}{2}+}} \lesssim \|u\|_{L^{\infty}_t H^1_x} \|v\|_{L^{\infty}_t H^1_x}.
$$
 
\end{lemma}
Finally, we deal with the tri-linear case, which is necessitated due to the failure of the appropriate  bilinear estimate. 
\begin{lemma}
\label{le:30}
For $0\leq\alpha<1/2$ and  $\gamma < \min(1- 2\alpha, 1/2)$, and $u,v,w$ smooth, 
$$
\|\mathcal{N} (T^{\ve_1,\ve_2} (u, v), w)\|_{X^+_{\gamma, -\frac{1}{2}+}} \lesssim \|u\|_{X^{\ve_1}_{0,\frac{1}{2}+}} \|v\|_{X^{\ve_2}_{0,\frac{1}{2}+}} \|w\|_{X^{\ve_3}_{0,\frac{1}{2}+}}.
$$
\end{lemma}

\begin{figure}[!h]
\begin{center}
\scalebox{0.50}{\includegraphics{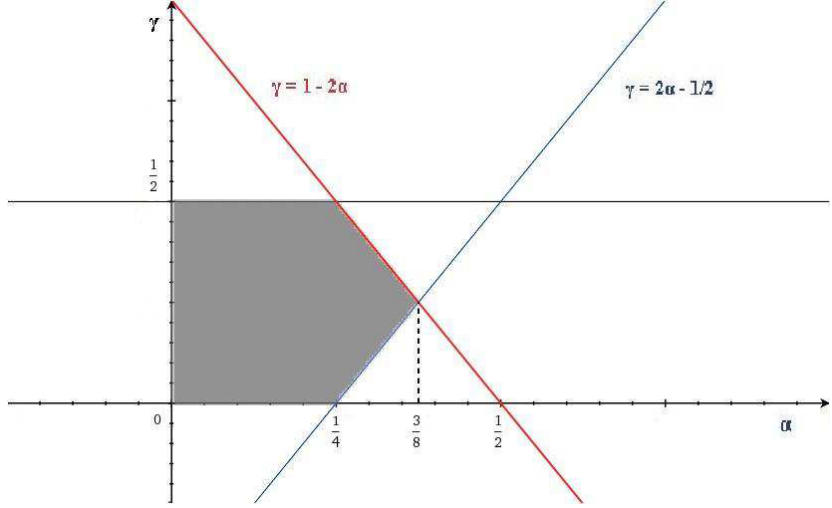}}
\end{center}
\caption{Permissible region for $(\alpha, \gamma)$}
\label{fig:region}
\end{figure}

\textbf{Remarks:} 
\begin{itemize}
\item From Figure~\ref{fig:region}, we note that $\gamma =0$ is permissible up to $\alpha<1/4$.  This leads to the case described in \cite{FS}.  The restriction $\ga>2\al -1/2$ comes from Lemma~\ref{le:20}.  It is easy to see from this graph where improvements can be made via the normal from method.

\item The restriction $\ga < 1-2\al$ results from Lemma~\ref{le:30}, and this is shown to be sharp in Section~\ref{sec:sharp}.  This leads to the restriction $\al <3/8$ instead of our original conjecture $\al<1/2$.
\end{itemize}

 \section{Proof of the bilinear and tri-linear estimates }

For the purposes of estimates, we   treat 
$\mathcal{N}(u,v) \sim \langle \nabla\rangle^{-\alpha}(\langle \nabla \rangle^{\alpha}u \langle \nabla \rangle^{\alpha} v)$. 

\subsection{Proof of Lemma \ref{le:20}} 
Let $\lambda_j = \tau_j - \ve_j \xi_j^2$ for $j=1,2$ where $\tau =\tau_1+ \tau_2$ and $\xi=\xi_1 + \xi_2$.  First we localize modulation $\tau -\ve \xi^2$ of functions $u$, $v$ by writing for example
$$
\widetilde{u} (\tau,\xi) = \sum_{k=0}^{\infty} \chi_{[2^k, 2^{k+1})} ( \langle \tau - \ve_1 \xi^2\rangle) \widetilde{u}(\tau,\xi).
$$
So in the following, we will assume that $\lambda_1 \sim L_1$, $\lambda_2 \sim L_2$ and $\tau - \xi^2 \sim L$ for some dyadic indices $L_1, L_2, L$.  At the end of the estimate, we will have the bound in terms of summable constants in all dyadic indices (e.g. $L_{\max}^{-\de/10}$ where $L_{\max} = \max(L, L_1, L_2)$).

We will show computations for the case $L_1 = L_{\max}$.  It will be clear that the other cases follow in a similar manner.  Applying the duality $(X^+_{s,b})^* = X^-_{s,b}$, we compute
\begin{align*}
\| N(u,v)\|_{X^+_{\ga, -\frac{1}{2}+\de}} &\sim \sup_{\|w\|_{X^-_{0,\frac{1}{2}-\de}}=1} \left| \int_{\rone\times \rone} \mathcal{N}(u,v) \, \langle \nabla \rangle^{\ga} w\, dx\, dt\right|\\
	&\sim \sup_{\|w\|_{X^-_{0,\frac{1}{2}-\de}}=1} \left| \int_{\tiny \begin{array}{l} \tau_1 + \tau_2 = \tau\\ \xi_1 + \xi_2 =\xi \end{array}} \frac{|\xi| \langle \xi_1\rangle^{\alpha} \langle \xi_2 \rangle^{\alpha}}{\langle \xi\rangle^{1+\alpha - \gamma}} \widetilde{u}(\tau_1, \xi_1) \widetilde{v}(\tau_2, \xi_2) \, \widetilde{w}(\tau,\xi)\, d\sigma \right|\\
	&\lesssim M_1 \sup_{\|w\|_{X^-_{0,\frac{1}{2}-\de}}=1} \left| \int_{\tiny \begin{array}{l} \tau_1 + \tau_2 = \tau\\ \xi_1 + \xi_2 =\xi \end{array}} \left[L_1^{\frac{1}{2}-\de} \langle \xi_1\rangle^{\gamma}\left|\widetilde{u}\right| \right]\, \left|\widetilde{v}\right|\, \, \left|\widetilde{w}\right|\, d\sigma \right|\\
	&\lesssim M_1 \sup_{\|w\|_{X^-_{0,\frac{1}{2}-\de}}=1} \|\langle\lambda_1\rangle^{\frac{1}{2}-\de} \langle \xi \rangle^{\gamma} \widetilde{u} \|_{L^2_{\tau} l^2_{\xi}} \left\|\mathcal{F}^{-1}_{\tau,\xi} \left|\widetilde{v}\right| \right\|_{L^4_{t,x}} \left\|\mathcal{F}^{-1}_{\tau,\xi} \left|\widetilde{w} \right| \right\|_{L^4_{t,x}}\\
	&\lesssim M_1 \|u\|_{X^{\gamma,\frac{1}{2}+\delta}} \|v\|_{X^{0,\frac{1}{2}+\delta}}
\end{align*}
where
\begin{equation}\label{m1}
M_1 \sim \sup_{\tiny\begin{array}{c}\xi_1+ \xi_2 = \xi\\ \xi\, \xi_1\, \xi_2 \neq 0\end{array}}  \frac{ \langle \xi_1\rangle^{\alpha-\gamma} \langle \xi_2 \rangle^{\alpha}}{\langle \xi\rangle^{\alpha - \gamma} L_{\max}^{\frac{1}{2}-\de}}.
\end{equation}

Note that we have used the embeddings $X^{\ve}_{0,1/2+}\subset X^{\ve}_{0,1/2-}\subset X^{\ve}_{0,3/8} \subset L^4_{t,x}$ to obtain the last inequality above.

It suffices to show that $M_1$ is bounded by summable constants in $L_{\max}$.  Let $N:= \max(\xi_1, \xi_2)$ and note that $\xi \leq 2N$ when $\xi_1 + \xi_2 = \xi$.  Also we note the following
$$
\lambda_1 + \lambda_2 = \tau - \xi^2 + \left[(\xi_1+ \xi_2)^2 -\ve_1 \xi_1^2 -\ve_2 \xi_2^2\right].
$$

Therefore, we must have $L_{\max} \gtrsim \left|(\xi_1+ \xi_2)^2 -\ve_1 \xi_1^2 -\ve_2 \xi_2^2\right|$.   

\textbf{Case 1.} When $\ve_1 = \ve_2 = -1$, we have $L_{max} \gtrsim N^2$. First if $\al \geq \ga$, then  $M_1 \lesssim N^{2\alpha -\ga} L_{\max}^{-\f{1}{2}+\de} \lesssim N^{2\al-\ga -1+4\de} L_{\max}^{-\de}$.  Therefore, we need to have $\ga > 2\al -1$ and appropriately small $\ga>0$.

Otherwise, if $\al < \ga$, then $M_1 \lesssim N^{\ga} L_{\max}^{\f{1}{2}-\de} \lesssim N^{\ga - 1+ 4\de} L_{\max}^{-\de}$, so $\ga <1$ would suffice.

\textbf{Case 2.} If $\ve_1 = \ve_2 = +1$, then we have $L_{\max} \gtrsim \xi_1 \xi_2$.  Then 
$$
M_1 \lesssim \frac{\langle \xi_1\rangle^{\alpha-\ga -\f{1}{2}+2\de}  \langle \xi_2\rangle^{\al -\f{1}{2}+2\de}}{\langle \xi\rangle^{\al-\ga} L_{\max}^{\de}}.
$$

If $\al\geq \ga$, then it suffices to require $\ga \geq 0$ and $\alpha <1/2$. 

If $\al < \ga$, then it suffices to require $\ga <1/2$ and $2\al -1 <\ga$.

\textbf{Case 3.} The remaining cases are either $\ve_1 = +1$, $\ve_2 = -1$ or $\ve_1 = -1$, $\ve_2 = +1$.  The first case gives $L_{\max} \gtrsim \xi \xi_2$ and the second gives $L_{\max} \gtrsim \xi \xi_1$.  So we have respectively
$$
M_1 \lesssim  \frac{ \langle \xi_1\rangle^{\alpha-\gamma} \langle \xi_2 \rangle^{\alpha-\f{1}{2}+2\de}}{\langle \xi\rangle^{\alpha - \gamma+\f{1}{2}-2\de} L_{\max}^{\de}} \qquad \textnormal{ or }\qquad
M_1 \lesssim  \frac{ \langle \xi_1\rangle^{\alpha-\gamma-\f{1}{2}+2\de} \langle \xi_2 \rangle^{\alpha}}{\langle \xi\rangle^{\alpha - \gamma+\f{1}{2}-2\de} L_{\max}^{\de}}.
$$
In both cases, if $\xi \sim N$, then it suffices to require $\gamma <1/2$.

If $\xi \ll N$, then both estimates give $M_1 \lesssim N^{2\al -\ga -\f{1}{2}+2\de} L_{\max}^{-\de}$.  Therefore we need to require $2\al -1/2 <\ga$.  We remark that this is the strongest bound which as appeared for this lemma.

Next, we prove Lemma \ref{le:25}. 
\subsection{Proof of Lemma \ref{le:25}} 
We will ignore the gain due to $\lambda^{1/2-}$ for this proof.
\begin{align*}
\|\mathcal{N} (u,v)\|_{L^2_T H^{\gamma}_x} &\sim \| \sum_{\xi_1 + \xi_2 = \xi} \frac{ |\xi|\langle \xi_1\rangle^{\alpha} \langle \xi_2 \rangle^{\alpha}}{\langle \xi\rangle^{1+\alpha - \gamma}}  \widehat{u}(\xi_1)  \widehat{v}(\xi_2)\|_{L^2_T l^2_{\xi}}\\
	&\hspace{-2.2cm}\sim \left\| \sum_{\xi_1 + \xi_2 = \xi} \frac{ |\xi|\langle \xi_1\rangle^{\alpha-1} \langle \xi_2 \rangle^{\alpha-\frac{1}{2}+\delta}}{\langle \xi\rangle^{1+\alpha - \gamma}}  [\langle \xi_1\rangle\widehat{u}(\xi_1) ] [\langle \xi_2\rangle^{\frac{1}{2}-\de}\widehat{v}(\xi_2)]\right\|_{L^2_T l^2_{\xi}}\\
	&\hspace{-2.2cm}\lesssim M_2 \left\|\left|\widehat{\langle \nabla \rangle u}\right| *_{\xi} \left|\widehat{\langle \nabla \rangle^{\frac{1}{2}-\de} v} \right| \right\|_{L^2_T l^2_{\xi}} \lesssim M_2\left\|\mathcal{F}^{-1}_{\xi}\left|\widehat{\langle \nabla \rangle u}\right|\right\|_{L^{\infty}_T L^2_{x}} \left\|\mathcal{F}^{-1}_{\xi} \left|\widehat{\langle \nabla \rangle^{\frac{1}{2}-\de} v}\right|\right\|_{L^{2}_T L^{\infty}_{x}}\\
	&\hspace{-2.2cm}\lesssim_{\de} M_2\|u\|_{L^{\infty}_t H^1_x} \|v\|_{L^2_T H^1_x} \lesssim_T M_2\|u\|_{L^{\infty}_t H^1_x} \|v\|_{L^{\infty}_T H^1_x(\mathbf{T})}
\end{align*}
where 
$$
M_2 \sim \sup_{\xi_1 + \xi_2 = \xi} \frac{ \langle \xi_1\rangle^{\alpha-1} \langle \xi_2 \rangle^{\alpha-\frac{1}{2}+\de}}{\langle \xi\rangle^{\alpha - \gamma}}.
$$

Note that we have used Sobolev embedding $H^{1/2+}_x(\mathbf{T}) \subset L^{\infty}_x (\mathbf{T})$ above.  To prove the desired estimate, we need to bound $M_2$ by an absolute constant.

If $\al \geq \ga$, then it suffices to have $\al <1/2$.

If $\al <\ga$, then it suffices to have $\ga <1/2$.

Lastly, we prove Lemma \ref{le:30}. 
\subsection{Proof of Lemma \ref{le:30}} 
In this proof, we will cover the cases when $\ve_1 =+1$ and $\ve_2 = -1$; or $\ve_1 = \ve_2 = +1$.  The remaining case $\ve_1 = \ve_2=-1$ is easier due to a faster decay in $\xi_1,\xi_2$, so it will not be argued here.
 
\textbf{Case 1.} First we consider the case where $\ve_1 =+1$, $\ve_2 = -1$.  Let $\lambda_j = \tau_j - \ve_j \xi_j^2$ for $j= 1,2,3,4$ where $\ve_4 =-\ve = -1$.   As in the proof of Lemma~\ref{le:20}, we localize modulations of each functions with respect to dyadic indices $L_1, L_2, L_3, L_4$ so that $\langle \tau_j -\ve_j \xi_j \rangle \sim L_j$ for $j=1,2,3,4$.  In the end, we will have an estimate in terms of a summable bound for $L_{\max}:= \max (L_1, L_2, L_3, L_4)$.

Let $\Gamma := \{ (\tau, \xi) \in \mathbf{R}^4 \times \mathbf{Z}^4: \tau_1 + \tau_2 + \tau_3 +\tau_4 =0, \, \xi_1 + \xi_2 + \xi_3 + \xi_4 =0\}$ and let $d\sigma$ be the inherited  measure on $\Gamma$.  Then
\begin{align*}
\|\mathcal{N} (T^{+,-} (u, v), w)\|_{X^+_{\gamma, -\frac{1}{2}+\de}} &\sim \sup_{\|z\|_{X^-_{0, \frac{1}{2}-\de}} =1} \left|\int_{\Gamma} a(\xi)\, \widetilde{u}(\tau_1, \xi_1)\, \widetilde{v}(\tau_2, \xi_2)\, \widetilde{w}(\tau_3, \xi_3) \,\widetilde{z}(\tau_4, \xi_4) \, d\sigma\right|
\end{align*}
where
$$
a(\xi) \sim \frac{\langle \xi_1\rangle^{\alpha}\langle \xi_3 \rangle^{\alpha}\langle \xi_4\rangle^{\gamma-\alpha}  }{\langle \xi_1 + \xi_2 \rangle \langle \xi_2 \rangle^{1-\alpha}} \qquad \qquad \textnormal{if } \xi_1 \xi_2 \xi_4 (\xi_1+ \xi_2)\neq 0
$$
and $a(\xi) = 0$ otherwise.   If $L_{\max} \sim L_1$ for instance, the integral above can be estimated as follows.

\begin{align*}
\int_{\Gamma} \left|a\, \widetilde{u}\, \widetilde{v}\, \widetilde{w} \,\widetilde{z}\right| \, d\sigma  &\lesssim \int_{\Gamma} \left|\frac{a}{L_1^{\frac{1}{2}+}}\, \lambda_1^{\frac{1}{2}+} \widetilde{u}\, \widetilde{v}\, \widetilde{w} \,\widetilde{z}\right| \, d\sigma \\
	&\lesssim \sup_{\xi} \left| \frac{a(\xi) \langle\xi_2\rangle^{\de} \langle \xi_3\rangle^{\de} \langle \xi_4\rangle^{\de}}{L_{\max}^{\frac{1}{2}-\de}}\right| \|L_1^{\frac{1}{2}+} u\|_{L^2_{t,x}} \| v_{\de} w_{\de} [L_4^{-2\de} z_{\de}] \|_{L^2_{t,x}}\\
	&\lesssim \sup_{\xi} \left| \frac{a(\xi)N^{3\de}}{L_{\max}^{\frac{1}{2}-\de}}\right| \|u\|_{X^{0,\frac{1}{2}+\de}} \|v_{\de}\|_{L^6_{t,x}} \|w_{\de}\|_{L^6_{t,x}} \|L_4^{-2\de} z_{\de}\|_{L^6_{t,x}} \\
	&\lesssim M_3 \|u\|_{X^{0,\frac{1}{2}+\de}}\|v\|_{X^{0,\frac{1}{2}+\de}}\|w\|_{X^{0,\frac{1}{2}+\de}}\|z\|_{X^{0,\frac{1}{2}-\de}}
\end{align*}
where Let $N := \max(|\xi_1|, |\xi_2|, |\xi_3|, |\xi_4|)$, $ u_{\de} := \mathcal{F}^{-1}_{\tau,\xi} \left[\langle \xi \rangle^{-\de} |\widetilde{u}|(\tau,\xi) \right]$ and 

$$
M_3:= \sup_{(\tau,\xi)\in \Gamma} \frac{\langle \xi_1\rangle^{\alpha}\langle \xi_3 \rangle^{\alpha}\langle \xi_4\rangle^{\gamma-\alpha} N^{3\de} }{\langle \xi_1 + \xi_2 \rangle \langle \xi_2 \rangle^{1-\alpha} L_{\max}^{\frac{1}{2}-\de}} 
$$  

Note that we have used $X^{0+,\frac{1}{2}+} \subset L^6_{t,x}$ for the last inequality.  Now it suffices to bound $M_3$ by a constant summable in $L_{\max}$.  First we observe the following two scenarios:
\begin{align}
\ve_3 = +1: \qquad &\sum_{j=1}^4 \lambda_j = -\xi_1^2 +\xi_2^2 -\xi_3^2 +\xi_4^2 = 2(\xi_1+\xi_2)(\xi_2 + \xi_3); \label{Lmax+}\\
\ve_3 = -1: \qquad &\sum_{j=1}^4 \lambda_j = -\xi_1^2 +\xi_2^2 +\xi_3^2 +\xi_4^2 = -2(\xi_2 \xi_3 + \xi_3 \xi_4 + \xi_4 \xi_2). \label{Lmax-}
\end{align}

  We split into the following cases for this estimate:

\textbf{Case 1A.} If $|\xi_1 + \xi_2| \gtrsim \max(|\xi_3|, |\xi_4|)$, then for $\al<1/2$ and $\ga <1/2$,
$$
M_3 \lesssim \frac{\langle \xi_1\rangle^{\alpha} N^{3\de}}{\langle \xi_1 + \xi_2 \rangle^{1-\ga} \langle \xi_2 \rangle^{1/2}L_{\max}^{\f{1}{2}-\de}}\lesssim \f{1}{L_{\max}^{\f{1}{2}-\de}}.
$$
So we are done.  Negation of Case 1A gives $|\xi_1 +\xi_2| \ll \max(|\xi_3|, |\xi_4|)$, which implies $\xi_3\sim \xi_4$ because of the relation $\xi_1+\xi_2+\xi_3+\xi_4 =0$.  The next case covers the possibility that $\xi_1, \xi_2$ may be large with opposite signs.

\textbf{Case 1B.}  Negation of Case 1A and also $\max(|\xi_1|, |\xi_2|) \sim N$.  Note that since $|\xi_1 + \xi_2| \ll N$, we must have $\xi_1 \sim \xi_2$.  Then
$$
M_3 \lesssim \f{ \langle \xi_3\rangle^{\ga} }{N^{1-2\al  -3\de} L_{\max}^{\f{1}{2}-\de}}.
$$

So we must have $\ga < 1- 2\al$.  This is where the upper bound in Lemma~\ref{le:30} for $\ga$ originates from.  We remark that this is completely necessary due to the cases such as
\begin{equation}\label{ce:1}
\xi_1 = N+1, \qquad \xi_2 = -N, \qquad \xi_3 = N,\qquad \xi_4 = -N-1.
\end{equation}

Note that if above holds, $L_{\max}$ does not have to be comparable $N$ in the case \eqref{Lmax+}, thus the bound for $M_3$ cannot be improved.  We have used \eqref{ce:1} to construct a counter-example for the cases $\ga >1-2\al$.   

By similar computations as above, the special case $\ga = 1-2\al$ can be shown to be true if $X^{0,1/2+} \subset L^6_{t,x}$ were  true.  However, this is an open conjecture of Bourgain (see \cite{Bour}) and it does not have a significant bearing on our conclusion, so we overlook this case.

\textbf{Case 1C.} Now the remaining case is when $\max(|\xi_1|, |\xi_2|) \ll N$.  Recall that $N\sim \xi_3 \sim \xi_4$.  This implies that $\xi_2 + \xi_3 \sim N$, so the case \eqref{Lmax+} gives that $L_{\max} \gtrsim N$.  The case \eqref{Lmax-} is even better since this gives $L_{\max} \sim N^2$.  So we take the lesser of these two bounds to estimate $M_3$ below.  Since $|\xi_1| \leq 2\max(|\xi_1 + \xi_2|, |\xi_2|)$, 
$$
M_3 \lesssim \f{\langle \xi_1\rangle^{\al}}{\langle \xi_1 + \xi_2\rangle \langle \xi_2 \rangle^{1-\al}} \f{N^{\ga+3\de}}{L_{\max}^{\f{1}{2}-\de}}
	\lesssim N^{\ga -\f{1}{2} +5\de} L_{\max}^{-\de}.
$$

So it suffices to require $\ga <1/2$.  This exhausts all cases for Case 1.\\

\textbf{Case 2.}  Now we consider the case where $\ve_1 = \ve_2 = +1$.  Following the same arguments as in the previous case, we have
$$
a(\xi) \sim \f{\langle \xi_3\rangle^{\al} \langle \xi_4\rangle^{\ga-\al}}{\langle \xi_1 \rangle^{1-\al} \langle \xi_2\rangle^{1-\al}} \qquad \qquad \textnormal{if } \xi_1 \xi_2 \xi_4 (\xi_1+ \xi_2)\neq 0
$$
and $a(\xi) = 0$ otherwise.  By the same series of estimates, it suffices to estimate $M_4$ by a constant summable in $L_{\max}$ where
$$
M_4:= \sup_{(\tau,\xi)\in \Gamma} \f{\langle \xi_3\rangle^{\al} \langle \xi_4\rangle^{\ga-\al} N^{3\de}}{\langle \xi_1 \rangle^{1-\al} \langle \xi_2\rangle^{1-\al} L_{\max}^{\f{1}{2}-\de}}.
$$

In this case, we have the following scenarios:
\begin{align}
\ve_3 = +1: \qquad &\sum_{j=1}^4 \lambda_j = -\xi_1^2 -\xi_2^2 -\xi_3^2 +\xi_4^2 = 2(\xi_1\xi_2 + \xi_3 [ \xi_1 + \xi_2]); \label{Lmax2+}\\
\ve_3 = -1: \qquad &\sum_{j=1}^4 \lambda_j = -\xi_1^2 -\xi_2^2 +\xi_3^2 +\xi_4^2 = 2(\xi_1 + \xi_3)(\xi_2 + \xi_3). \label{Lmax2-}
\end{align}

\textbf{Case 2A.} If $|\xi_1 \xi_2| \gtrsim N$, then $M_4 \lesssim N^{\ga + \al-1 +3\de}L_{\max}^{-1/2+\de}$.  So we are done since $\ga<1/2$ and $\al<1/2$.

\textbf{Case 2B.} The remaining cases must have $|\xi_1 \xi_2| \ll N$, which implies $\xi_3 \sim \xi_4 \sim N$.  Then the case \eqref{Lmax2+} gives $L_{\max} \gtrsim N$.  On the other hand, the case \eqref{Lmax2-} gives $L_{\max} \gtrsim N^2$.  We use the lesser of these two to estimate

$$
M_4 \lesssim \frac{N^{\ga+ 3\de}}{L_{\max}^{\f{1}{2}-\de}} \lesssim N^{\ga -\f{1}{2} + 5\de} L_{\max}^{-\de}.
$$

Since this is summable for $\ga <1/2$, we are done.

\subsection{Failure of Lemma~\ref{le:30} if $\ga >1-2\al$}\label{sec:sharp}

In this section, we construct an explicit counter-example to show that the following estimate fails
\begin{equation}\label{counterexample}
\| \mathcal{N}(T^{+,-}(u,v),w)\|_{X^+_{\ga, -\f{1}{2}+\de}} \leq C_{\de} \|u\|_{X^{+}_{0,\f{1}{2}+\de}} \|v\|_{X^{-}_{0,\f{1}{2}+\de}} \|w\|_{X^{+}_{0,\f{1}{2}+\de}}.
\end{equation}

Given $\eta \in \mathcal{S}_t(\mathbb{R})$ and $N\gg 1$, let $u$, $v$, $w$ be defined as follows:
$$
u(t,x) := \eta(t) e^{i(N+1)^2t + i(N+1)x}; \quad v(t,x) = \eta(t) e^{-iN^2t - iNx}; \quad w(t,x)=\eta(t) e^{iN^2t+ iNx}.
$$

First, we remark that the right side of \eqref{counterexample} is equal to $C \|\eta\|^3_{H_t^{1/2+\de}}$, where $C$ is independent of $N$.
Substituting these functions to \eqref{normal}, we obtain
$$
T^{+,-}(u,v)(t,x) = C_{\al} \eta^2(t) \f{\langle N+1\rangle^{\al} \langle N\rangle^{\al} e^{i(2N + 1)t +ix}}{N\left[ \langle N+1\rangle - \langle N\rangle\right] + \langle N\rangle -  \sqrt{2}}.
$$
Recall $\mathcal{N}(u,v) = |\nabla|\langle \nabla\rangle^{-1-\al} \left[ \langle \nabla\rangle^{\al} u \langle \nabla \rangle^{\al} v\right]$.  Then writing $\phi = \eta^3$, we have
$$
\mathcal{N}(T^{+,-}(u,v),w) = C_{\al} \phi (t) \f{|N+1| \langle N\rangle^{2\al}e^{i(N + 1)^2 t +  i(N+1)x}}{ \langle N+1 \rangle (N\left[ \langle N+1\rangle - \langle N\rangle\right] + \langle N\rangle -  \sqrt{2})} .
$$
 Then
 \begin{equation}\label{eq:integ}
 \| \mathcal{N}(T^{+,-}(u,v),w)\|_{X^+_{\ga, -\f{1}{2}+\de}} = C(N, \al, \ga) \left( \int_{\mathbb{R}}\f{|\widehat{\phi}|^2 (\tau - (N+1)^2)}{\langle \tau - (N+1)^2\rangle^{1-2\de}} \, d\tau \right)^{\f{1}{2}}
 \end{equation}
where 
$$
C(N, \al, \ga) := C \f{|N+1| \langle N\rangle^{2\al} }{ \langle N+1 \rangle^{1-\ga} (N\left[ \langle N+1\rangle - \langle N\rangle\right] + \langle N\rangle -  \sqrt{2})}.
$$

Note that the integral in \eqref{eq:integ} becomes independent of $N$ after a change of variable.  Also, for large $N$, $C(N,\al,\ga) \sim N^{2\al + \ga -1}$.  Since the right side of \eqref{counterexample} is independent of $N$, the trilinear estimate cannot hold if $2\al+\ga >1$.

\end{document}